\documentclass{amsart}
\usepackage{amsthm,amssymb,verbatim,color}
\usepackage{graphicx}
\usepackage{pgf,tikz}

\newtheorem{theorem}{Theorem}[section]
\newtheorem{corollary}[theorem]{Corollary}
\newtheorem{lemma}[theorem]{Lemma}
\newtheorem{proposition}[theorem]{Proposition}
\theoremstyle{definition}
\newtheorem{definition}[theorem]{Definition}
\newtheorem{remark}[theorem]{Remark}
\newtheorem{example}[theorem]{Example}
\newtheorem{convention}[theorem]{Convention}

\newcommand{\X}{{\mathbb{X}}}
\newcommand{\Y}{{\mathbb{Y}}}

\newcommand{\Z}{{\mathbb{Z}}}
\newcommand{\N}{{\mathbb{N}}}
\newcommand{\Q}{{\mathbb{Q}}}

\newcommand{\bbP}{{\mathbb{P}}}
\newcommand{\fm}{{\mathfrak{m}}}

\newcommand{\pmpn}{\mathbb{P}^{m}\!\times\mathbb{P}^{n}}

\newcommand{\HF}{\operatorname{HF}}
\newcommand{\im}{\operatorname{Im}}

\begin{document}

\title{The K\"{a}hler Different
of a Set of Points in~$\pmpn$}

\author{Tran N.K. Linh}
\address{Department of Mathematics\\
	University of Education, Hue University\\
	34 Le Loi, Hue, Vietnam}
\email{tnkhanhlinh@hueuni.edu.vn}

\author{Le N. Long}
\address{Fakult\"{a}t f\"{u}r Informatik und Mathematik \\
	Universit\"{a}t Passau, D-94030 Passau, Germany \newline
	\hspace*{.5cm} \textrm{and} Department of Mathematics,
	University of Education, Hue University\\
	34 Le Loi, Hue, Vietnam}
\email{lelong@hueuni.edu.vn}

\author{Nguyen T. Hoa}
\address{Department of Mathematics\\
	University of Education, Hue University\\
	34 Le Loi, Hue, Vietnam}
\email{nguyenthihoa2252000@gmail.com}

\author{Nguyen T.P. Nhi}
\address{Department of Mathematics\\
	University of Education, Hue University\\
	34 Le Loi, Hue, Vietnam}
\email{nhimeo293@gmail.com}

\author{Phan T.T. Nhan}
\address{Department of Mathematics\\
	University of Education, Hue University\\
	34 Le Loi, Hue, Vietnam}
\email{nhan9715@gmail.com}

\date{\today}

\begin{abstract}
Given an ACM set $\X$ of points in a multiprojective space $\pmpn$
over a field of characteristic zero, we are interested in studying
the K\"ahler different and the Cayley-Bacharach property for $\X$.
In $\bbP^1\times \bbP^1$, the Cayley-Bacharach property agrees with
the complete intersection property and it is characterized
by using the K\"ahler different. However, this result fails to hold
in $\pmpn$ for $n>1$ or $m>1$.
In this paper we start an investigation of the K\"ahler different
and its Hilbert function and then prove that $\X$ is
a complete intersection of type $(d_1,...,d_m,d'_1,...,d'_n)$
if and only if it has the Cayley-Bachrach property and
the K\"ahler different is non-zero at a certain degree.
When $\X$ has the $(\star)$-property, we characterize
the Cayley-Bacharach property of $\X$ in terms of its components
under the canonical projections.
\end{abstract}



\maketitle

%
%

\section{Introduction} \label{intro}

Let $\X$ be a finite set of points in the multiprojective space
$\pmpn$ over a field $K$ of characteristic zero, let
$I_\X \subseteq S:=K[X_0,...,X_m,Y_0,...,Y_n]$
be the bihomogeneous vanishing ideal of~$\X$, and let
$R_\X =S/I_\X$ be the bigraded coordinate ring of~$\X$.
The set $\X$ is called \textit{arithmetically Cohen-Macaulay (ACM)}
if $R_\X$ is a Cohen-Macaulay ring,
and $\X$ is called a \textit{complete intersection of type
$(d_1,...,d_m,d'_1,...,d_n)$} if $I_\X$ is generated
by a bihomogeneous regular sequence $\{F_1,...,F_m,G_1,...,G_n\}$
with $\deg(F_i)=(d_i,0)$ for $i=1,...,m$
and $\deg(G_j)=(0,d'_j)$ for $j=1,...,n$.
The study of special classes of finite sets of points such as
ACM sets of points, complete intersections, etc.
in a multiprojective space is a very active field of research
and has been attracted by many authors.
For instance, the work on finding a classification of ACM
set of points includes
\cite{GJ2019, GuVT2004, GuVT2008, GuVT2008b, Mar2009, Tuy2003}
and the work on complete intersections includes
\cite{CN2020, GKLL2017, GLL2018, KL2017}.

Obviously, every complete intersection of type
$(d_1,...,d_m,d'_1,...,d_n)$ is ACM. It is a subject of research
to understand when $\X$ is a complete intersection of type
$(d_1,...,d_m,d'_1,...,d_n)$.
One of the classical tools for studying the complete intersection
property is the K\"ahler different (see \cite{KL2017, Kun1986, SS1975}).
When $\X$ is ACM, we may assume that $R_o:=K[X_0,Y_0]$
is a Noetherian normalization of $R_\X$ and define the
\textit{K\"ahler different} $\vartheta_\X$ of $\X$ or of
the bigraded algebra $R_\X/R_o$ which is known
as the initial Fitting ideal of the K\"ahler differential module
of~$R_\X/R_o$.
In the case $m=n=1$, \cite[Proposition~7.3]{GKLL2017}
shows that an ACM set $\X$ is a complete intersection
of type $(d_1,d'_1)$ if and only if $\vartheta_\X$ contains
no separators for $\X$ of degree less than $(2r_{\X_1},2r_{\X_2})$,
where $\X_i =\pi_i(\X)$ and $r_{\X_i}$ is the regularity index
of the Hilbert function of $\X_i$ for $i=1,2$ and
$\pi_1: \pmpn\rightarrow \bbP^m$ and
$\pi_2: \pmpn\rightarrow \bbP^n$ are the canonical projections,
which in turn is equivalent to the condition
that $\X$ has the Cayley-Bacharach property.
Here, we say that $\X$ has the \textit{Cayley-Bacharach property}
if the Hilbert function of $\X\setminus \{p\}$ is independent
of the choice of $p\in \X$. A nice history about the study
of the Cayley-Bacharach property of a finite set of points in
the projective space can be found in \cite{KLR2019}.
Notice that the above result of~\cite{GKLL2017} does not
hold true in general, for instance when $m>1$ or $n>1$
as Example~\ref{Exam-S4-13} shows. But if $\X\subseteq \pmpn$ is
a complete intersection of type $(d_1,...,d_m,d'_1,...,d_n)$,
then it still has the Cayley-Bacharach property and
$\vartheta_\X$ contains no separators for $\X$ of degree
less than $(2r_{\X_1},2r_{\X_2})$.
It is natural to ask which additional conditions
make an ACM set of points $\X$ with Cayley-Bacharach property
being a complete intersection of type $(d_1,...,d_m,d'_1,...,d_n)$.

Working on this question, in this paper we prove
the following result.
\begin{theorem}[Theorem~\ref{Thm-S4-10}]\label{TheoremA}
	For a set $\X$ of $s$ distinct points in $\pmpn$,
	the following are equivalent.
	\begin{enumerate}
		\item[(a)] $\X = CI(d_1,...,d_m,d'_1,...,d'_n)$ for some
		positive integers $d_i,d'_j\ge 1$.
		\item[(b)] $\X = \X_1\times \X_2$ has the Cayley-Bacharach
		property and $\HF_{\vartheta_{\X}}(r_{\X_1},r_{\X_2})\ne 0$.
	\end{enumerate}
\end{theorem}
Also, when $\X$ satisfies the $(\star)$-property
(see \cite[Definition~4.2]{GuVT2004}), we look
closely at the Cayley-Bacharach property for $\X$.
If we write $\X_1 =\pi_1(\X) =\{q_1,...,q_{s_1}\}\subseteq\bbP^m$
and $\X_2=\pi_2(\X)=\{q'_1,...,q'_{s_2}\}\subseteq\bbP^m$
and put
$$
W_i := \pi_2(\pi_1^{-1}(q_i)\cap \X) \subseteq \X_2,
\quad
V_j := \pi_1(\pi_2^{-1}(q'_j)\cap \X) \subseteq \X_1
$$
for $i=1,...,s_1$ and $j=1,...,s_2$, then we obtain
the following characterization of
the Cayley-Bacharach property for $\X$.
\begin{theorem}[Theorem~\ref{Thm-S5-02}]\label{TheoremB}
Suppose that $\X\subseteq\pmpn$ has the $(\star)$-property.
Then $\X$ has the Cayley-Bacharach property if and only if
the following conditions are satisfied:
\begin{enumerate}
	\item[(a)]  $V_1,...,V_{s_2}$ are Cayley-Bacharach schemes in $\bbP^m$
	and $r_{V_1}=\cdots = r_{V_{s_2}}$;
	\item[(b)] $W_1,...,W_{s_1}$ are Cayley-Bacharach schemes in $\bbP^n$
	and $r_{W_1}=\cdots = r_{W_{s_1}}$.
\end{enumerate}
\end{theorem}
Using Theorem~\ref{TheoremB}, in $\bbP^1\times\bbP^n$
we can drop the condition $\X=\X_1\times\X_2$ in part (b)
of Theorem~\ref{TheoremA} and get the following consequence.
\begin{theorem}[Corollary~\ref{Cor-S5-06}]
Suppose that $\X\subseteq\bbP^1\times \bbP^n$ has the $(\star)$-property.
Then $\X = CI(d_1,d'_1,...,d'_n)$ for some positive integers
$d_1,d'_1,...,d'_n\ge 1$
if and only if $\X$ has the Cayley-Bacharach property
and $\HF_{\vartheta_\X}(d_1-1,r_{\X_2})\ne 0$.
\end{theorem}

The paper is organized as follows. In Section~2 we fix the notation
and recall the definitions of the border of the Hilbert function of $\X$ and
the K\"ahler differential modules $\Omega^1_{R_\X/K}$ and $\Omega^1_{R_\X/R_o}$.
In particular, we use a presentation of $\Omega^1_{R_\X/K}$ (see Theorem~\ref{Thm-S2-08})
and its relation with $\Omega^1_{R_\X/R_o}$ to give a formula for
the Hilbert function of $\Omega^1_{R_\X/R_o}$ when $\X$ is ACM (see Proposition~\ref{Prop-S2-11}). In Section~3 we take a closed look at the
K\"ahler different $\vartheta_{\X}$ of an ACM set of points $\X$ in~$\pmpn$.
We provide several basic properties of the Hilbert function of $\vartheta_{\X}$
and its border. Section~4 contains the first main result (Theorem~\ref{Thm-S4-10})
which characterize $\X = CI(d_1,...,d_m,d'_1,...,d'_n)$ using the K\"ahler different
and the Cayley-Bacharach property. In this special case we describe explicitly
the Hilbert function of $\vartheta_{\X}$ and its border (see Proposition~\ref{Prop-S4-05}
and Corollary~\ref{Cor-S4-05}). In the final section, we restrict our attention
to the finite sets of points in $\pmpn$ having the $(\star)$-property.
In this setting, we relate the degree of a point $q_i\times q'_j\in \X$ to
degrees of points in $W_i$ and $V_j$ (see Proposition~\ref{Prop-S5-01}).
This enables us to prove a characterization of the Cayley-Bacharach property of $\X$
(see Theorem~\ref{Thm-S5-02}) and derive some consequences in $\bbP^1\times \bbP^n$
(see Proposition~\ref{Prop-S5-05} and Corollary~\ref{Cor-S5-06}).
All examples in this paper were calculated using the computer algebra system
ApCoCoA~\cite{ApC}.

\bigbreak
\section{The K\"ahler Differential Modules} \label{Sec2}

Let $K$ be a field of characteristic zero, let $m,n\ge 1$
be positive integers. For $(i_1,j_1), (i_2,j_2)\in \Z^2$,
we write  $(i_1,j_1) \preceq (i_2,j_2)$ if $i_1\le j_1$
and $i_2\le j_2$.  The bigraded coordinate ring of $\pmpn$
is the polynomial ring $S=K[X_0,\dots,X_m,Y_0,\dots,Y_n]$
equipped with the $\Z^2$-grading defined by
$\deg(X_0)=\cdots=\deg(X_m) = (1,0)$ and
$\deg(Y_0)=\cdots=\deg(Y_n) = (0,1)$.
For $(i,j)\in \Z^2$, we let $S_{i,j}$
be the homogeneous component of degree $(i,j)$ of~$S$,
i.e., the $K$-vector space with basis
$$
\{X_0^{\alpha_0}\cdots X_m^{\alpha_m}
\cdot Y_0^{\beta_0} \cdots Y_n^{\beta_n}
\mid \textstyle{\sum_{k=0}^m} \alpha_k = i,\,
\textstyle{\sum_{k=0}^n} \beta_k = j,\,
\alpha_k,\beta_k\in\Z \,\}.
$$

Given an ideal $I\subseteq S$, we set
$I_{i,j}:=I\cap S_{i,j}$ for all
$(i,j) \in \Z^2$.
The ideal $I$ is called \textit{bihomogeneous}
if $I = \bigoplus_{(i,j)\in \Z^2} I_{i,j}$.
If~$I$ is a bihomogeneous ideal of~$S$
then the quotient ring $S/I$ also inherits the structure
of a bigraded ring via
$(S/I)_{i,j}:= S_{i,j}/I_{i,j}$
for all $(i,j)\in \Z^2$.

A finitely generated $S$-module $M$ is
a {\it bigraded $S$-module} if it has a direct
sum decomposition
\[
M = \bigoplus_{(i,j) \in \Z^2} M_{i,j}\]
with the property that
$S_{(i_1,j_1)}M_{(i_2,j_2)}\subseteq M_{i_1+i_2,j_1+j_2}$
for all $(i_1,j_1),(i_2,j_2) \in \Z^2$.

\begin{definition}
Let $M$ be a finitely generated bigraded $S$-module.
The {\it Hilbert function} of~$M$ is the numerical function
$\HF_{M}: \Z^2 \rightarrow \N$ defined by
\[
\HF_{M}(i,j) := \dim_K M_{i,j}
\quad \mbox{for all $(i,j)\in \Z^2$}.
\]
In particular, for a bihomogeneous ideal $I$ of $S$,
the Hilbert function of $S/I$ satisfies
\[
\HF_{S/I}(i,j) := \dim_k (S/I)_{i,j}
= \dim_k S_{i,j} - \dim_k I_{i,j}
\quad \mbox{for all $(i,j)\in \Z^2$}.
\]
\end{definition}

If~$M$ is a finitely generated bigraded $S$-module
such that $\HF_M(i,j) =0$ for
$(i,j) \nsucceq (0,0)$,
we write the Hilbert function of~$M$ as an infinite matrix,
where the initial row and column are indexed by~0.

A point in the space $\pmpn$ has the form
$$
p= [a_0:a_1:\dots:a_m] \times
[b_0:b_1:\dots:b_n] \in \pmpn
$$
where $[a_0:a_1:\dots:a_m] \in \bbP^m$
and $[b_0:b_1:\dots:b_n]\in\bbP^n$.
Its vanishing ideal is the bihomogeneous prime ideal
of the form
$$
I_p = \langle \ell_1,\dots,\ell_m,
\ell'_1,\dots, \ell'_n \rangle \subseteq S
$$
where $\deg(\ell_i) = (1,0)$ and $\deg(\ell'_j) = (0,1)$
for $1\le i\le m, 1\le j \le n$.

\begin{definition}
Let $s\ge 1$ and let
$\X =\{p_1,\dots, p_s\}$ be a set of $s$ distinct points
in~$\pmpn$. The \textit{bihomogeneous vanishing ideal} of $\X$
is given by  $I_\X = I_{p_1}\cap\cdots\cap I_{p_s}$
and its \textit{bigraded coordinate ring} is $R_\X = S/I_\X$.
\end{definition}

In what follows, let $\X =\{p_1,\dots, p_s\}$ be a set of
$s$ distinct points in~$\pmpn$, and let $x_i$ and $y_j$ denote
the images of $X_i$ and $Y_j$ in $R_\X$ for $0\le i\le m$
and $0\le j\le n$. We write $\HF_\X$ for the Hilbert function
of $R_\X$ and call it the Hilbert function of $\X$.
It is worth to noting here that a bihomogeneous element
is a zerodivisor of $R_\X$ if and only if
it vanishes at some points of $\X$.

\begin{convention}\label{Conv-S2-03}
Given the canonical projections $\pi_1: \pmpn\rightarrow\bbP^m$
and $\pi_2: \pmpn\rightarrow\bbP^n$, we let
$\X_1 = \pi_1(\X)$, $s_1 = |\X_1|$, $\X_2 = \pi_2(\X)$,
and $s_2 =|\X_2|$. The set $\X_1$ has its homogeneous vanishing ideal
$I_{\X_1}\subseteq K[X_0,...,X_m]$ and its homogeneous
coordinate ring $R_{\X_1}=K[X_0,...,X_m]/I_{\X_1}$.
Similarly, $\X_2$ has its homogeneous vanishing ideal
$I_{\X_2}\subseteq K[Y_0,...,Y_n]$ and its homogeneous
coordinate ring $R_{\X_2}=K[Y_0,...,Y_n]/I_{\X_2}$.
\end{convention}

Notice that there exists a linear
form $\ell\in K[X_0,...,X_m]$ such that
$\ell$ does not vanish at any point of $\X_1$.
Analogously, we find a linear form
$\ell' \in K[Y_0,...,Y_n]$ which does not
vanish at any point of $\X_2$.
It follows that $\overline{\ell}, \overline{\ell}' \in R_\X$ are
non-zerodivisors (see also e.g. \cite[Lemma~1.2]{GuVT2004}).
As a consequence of this fact and \cite[Proposition 1.9]{SVT2006}
and \cite[Proposition 4.6]{Tuy2002},
we get several basis properties of the Hilbert function of~$\X$.

\begin{proposition}\label{Prop-S2-04}
Let $(i,j)\in \Z^2$ with $(i,j) \succeq (0,0)$.
\begin{enumerate}
	\item[(a)] We have
		$\HF_\X(i,j) \le \min\{\HF_\X((i+1,j),\HF_\X((i,j+1)\}\le s$.

	\item[(b)] If $\HF_\X(i,j) = \HF_\X(i+1,j)$
		then $\HF_\X(i,j) = \HF_\X(i+2,j)$.
		Also, $\HF_{\X}(i,j)=\HF_\X(s_1-1,j)$ for
		$i\ge s_1-1$ and $j<s_2-1$.

	\item[(c)] If $\HF_\X(i,j) = \HF_\X(i,j+1)$
	then $\HF_\X(i,j) = \HF_\X(i,j+2)$.
	Also, $\HF_{\X}(i,j)=\HF_\X(i,s_2-1)$ for
	$i< s_1-1$ and $j\ge s_2-1$.

	\item[(d)] We have $\HF_{\X}(i,j)=s$ for all $(i,j)\succeq(s_1-1,s_2-1)$.
\end{enumerate}
\end{proposition}

For $k,l \in \N$ set
$
\nu_k := \min\{ i\in\N \mid \HF_\X(i,k)= \HF_\X(i+1,k)\}
$
and
$
\varrho_l := \min\{j\in\N \mid \HF_\X(l,j)= \HF_\X(l,j+1)\}.
$
Let
$\nu:=\sup\{\nu_k \mid k\in \N \}$ and
$\varrho:=\sup\{\varrho_l \mid l \in \N\}$.
In view of Proposition~\ref{Prop-S2-04}, we have
$(\nu,\varrho)\preceq (s_1-1,s_2-1)$. Especially,
$(\nu,\varrho)=(s_1-1,s_2-1)$ if $m=n=1$.
Moreover, the tuple $(\nu,\varrho)$ can be described
by the following lemma.

\begin{lemma}\label{Lem-S2-05}
Let $k,l\in \N$.
If $\HF_\X(i,k)= \HF_\X(i+1,k)$ then $\HF_\X(i,k+1)= \HF_\X(i+1,k+1)$;
and if $\HF_\X(l,j)= \HF_\X(l,j+1)$ then
$\HF_\X(l+1,j)= \HF_\X(l+1,j+1)$.
In particular, we have
$(\nu,\varrho)=(r_{\X_1},r_{\X_2})$, where
$r_{\X_k}$ is the regularity index of $\HF_{\X_k}$ for $k=1,2$.
\end{lemma}
\begin{proof}
As in the argument before the Proposition~\ref{Prop-S2-04},
we find $\ell\in S_{1,0}$ and $\ell'\in S_{0,1}$ such that
their images $\bar{\ell}, \bar{\ell}'$ in $R_\X$ are non-zerodivisors.
Then we have
$$
\begin{aligned}
\HF_\X(i,k+1)&= \dim_K((R_\X)_{i,k}\cdot (R_\X)_{0,1})
= \dim_K(\bar{\ell}\cdot (R_\X)_{i,k}\cdot (R_\X)_{0,1})\\
&= \dim_K((R_{\X})_{i+1,k}\cdot (R_\X)_{0,1})
= \HF_\X(i+1,k+1),
\end{aligned}
$$
where the second equality follows from the fact that
$\bar{\ell}\in (R_\X)_{1,0}$ is a non-zerodivisor of $R_{\X}$
and the third equality induces by assumption that
$\HF_\X(i,k)= \HF_\X(i+1,k)$.
Analogously, by using the non-zerodivisor $\bar{\ell}'\in (R_\X)_{0,1}$,
we have $\HF_\X(l+1,j)= \HF_\X(l+1,j+1)$ when $\HF_\X(l,j)= \HF_\X(l,j+1)$.
Consequently, we get $\nu_k\ge \nu_{k+1}$ for all $k\in \mathbb{N}$
and $\varrho_l\ge \varrho_{l+1}$ for all $l\in \mathbb{N}$, and
hence $\nu=\nu_0=r_{\X_1}$ and $\varrho = \varrho_0=r_{\X_2}$.
\end{proof}

The lemma leads us to the following definition, which
agrees with \cite[Definition~4.9]{Tuy2002}
if $(\nu,\varrho)=(s_1-1,s_2-1)$.

\begin{definition}
Let $r_{\X_1}$, $r_{\X_2}$ be regularity indices of $\HF_{\X_1}$
and $\HF_{\X_2}$, respectively. The pair $B_\X=(B_C, B_R)$, where
$$
B_C=(\HF_\X(r_{\X_1},0), \HF_\X(r_{\X_1},1),\dots,\HF_\X(r_{\X_1},r_{\X_2}))
$$
and
$$
B_R=(\HF_\X(0,r_{\X_2}),\HF_\X(1,r_{\X_2}),\dots, \HF_\X(r_{\X_1},r_{\X_2})),
$$
is called the \textit{border of the Hilbert function} of~$\X$.
\end{definition}

\begin{example}\label{Exam-S2-07}
Let $K=\Q$, let $\X=\{p_1,...,p_9\}$ be a set of nine points
in~$\bbP^2\times\bbP^2$ given by $p_1=q_1\times q_1$,
$p_2=q_1\times q_2$, $p_3=q_1\times q_3$, $p_4=q_1\times q_4$,
$p_5=q_2\times q_1$, $p_6=q_2\times q_2$, $p_7=q_2\times q_3$,
$p_8=q_3\times q_1$ and $p_9 = q_3\times q_2$,
where $q_1 =(1:0:0)$, $q_2=(1:1:0)$, $q_3=(1:0:1)$,
$q_4=(1:1:1)$ in $\bbP^2$.
Then $\X_1=\{q_1,q_2,q_3\}$, $s_1=3$,
$\X_2 = \{q_1,q_2,q_3,q_4\}$ and $s_2=4$.
The Hilbert function of $\X$ is given by
$$
\HF_\X = \begin{bmatrix}
1&3&4&4&\cdots\\
3&8&9&9&\cdots\\
3&8&9&9&\cdots\\
3&8&9&9&\cdots\\
\vdots&\vdots&\vdots&\vdots&\ddots
\end{bmatrix},
$$
and so $r_{\X_1}=1$ and $r_{\X_2}=2$.
The border of the Hilbert function of~$\X$
is given by $B_\X = ((3,8,9),(4,9))$.
In this case we have $r_{\X_1}<2=s_1-1$ or $r_{\X_2} <3=s_2-1$,
and $\HF_\X(i,j)=s=6$ for all $(i,j)\succeq (r_{\X_1},r_{\X_2})$.
\end{example}

In the bigraded enveloping algebra $R_\X \otimes_K R_\X$
we have the bihomogeneous ideal $J=\ker(\mu)$,
where $\mu:\ R_\X\otimes_K R_\X \rightarrow R_\X$
is the bihomogeneous $R_\X$-linear map given by
$\mu(f\otimes g)=fg$.
The bigraded $R_\X$-module $\Omega^1_{R_\X/K}= J/J^2$
is called the {\it module of K\"{a}hler differentials}
of~$R_\X/K$. The bihomogeneous $K$-linear map
$d_{R_\X/K}:R_\X \rightarrow \Omega^1_{R_\X/K}$
given by $f\mapsto f\otimes 1-1\otimes f+J^2$ satisfies
the universal property. We call $d$
the \textit{universal derivation} of $R_\X/K$.
More generally, for any bigraded $K$-algebra $T/R$ we
can define in the same way the K\"{a}hler differential module
$\Omega^1_{T/R}$, and the universal derivation of $T/R$
(cf. \cite[Section~2]{Kun1986}).
Note that
$$
\Omega^1_{S/K} = \textstyle{\bigoplus\limits_{i=0}^m} SdX_i
\oplus \textstyle{\bigoplus\limits_{j=0}^n} SdY_j
\cong S^{m+1}(-1,0)\oplus S^{n+1}(0,-1)
$$
and $\Omega^1_{R_\X/K} =
\langle dx_i, dy_j \mid 0\le i\le m, 0\le j\le n \rangle_{R_\X}$.
Especially, the Hilbert function of $\Omega^1_{R_\X/K}$ can be
computed by using the following theorem
(see \cite[Theorem 3.5]{GLL2018}).

\begin{theorem} \label{Thm-S2-08}
Let $\Y$ be the subscheme of $\pmpn$ defined by
the bihomogeneous ideal $I_\Y = I_{p_1}^2\cap\cdots\cap I_{p_s}^2$.
There is an exact sequence of bigraded $R_\X$-modules
$$
0\longrightarrow I_\X/I_\Y \longrightarrow
R_{\X}^{m+1}(-1,0)\oplus R_{\X}^{n+1}(0,-1)
\longrightarrow \Omega^1_{R_\X/K}\longrightarrow 0.
$$
In particular, for $(i,j)\in \Z^2$, we have
$$
\HF_{\Omega^1_{R_\X/K}}(i,j) = (m+1)\HF_\X(i-1,j)+(n+1)\HF_\X(i,j-1)
+\HF_{\X}(i,j)-\HF_{\Y}(i,j).
$$
\end{theorem}

Notice that $R_\X$ has the Krull dimension 2,
but $1\leq \mathrm{depth}(R_\X)\leq 2$ (see \cite[Section~2]{Tuy2003}).
In case $\mathrm{depth}(R_\X)$ attains the maximal value,
we have the following notion.

\begin{definition}
We say that $\X$ is \textit{arithmetically Cohen-Macaulay (ACM)}
if we have $\mathrm{depth}(R_\X)=2$.
\end{definition}

When $\X$ is ACM, then there exist
two linear form $\ell\in S_{1,0}$, $\ell' \in S_{0,1}$
such that $\overline{\ell}$ and $\overline{\ell}'$
give rise to a regular sequence in $R_\X$
(see \cite[Proposition~3.2]{Tuy2003}).
After a change of coordinates, we can assume that
$\ell=X_0$ and $\ell'=Y_0$, so that
$x_0,y_0$ form a regular sequence in~$R_\X$.
In this case we set $R_o:=K[x_0,y_0]$. Then
$$
R_\X =S/I_\X = R_o[x_1,...,x_n,y_1,...,y_m]
$$
is a finitely generated, bigraded $R_o$-module,
and the monomorphism  $R_o\hookrightarrow R_\X$
defines a Noetherian normalization.

\begin{remark}\label{Rem-S2-10}
	Note that the \textit{Euler derivation} of $R_\X/K$ is
	given by $\epsilon: R_\X \rightarrow R_\X, f\mapsto (i+j)f$
	for $f\in (R_\X)_{i,j}$ (see \cite[Section~1]{Kun1986}).
	Set $\fm := \langle x_0,...,x_m,y_0,...,y_n\rangle_{R_\X}$.
	By the universal property of $\Omega^1_{R_\X/K}$,
	this induces a bihomogeneous surjective $R_\X$-linear map
	$\gamma: \Omega^1_{R_\X/K}\rightarrow \fm$
	with $\gamma(dx_i)=x_i$ and $\gamma(dy_j)=y_j$
	for all $i,j$.
	In particular, $\mathrm{Ann}_{R_\X}(dx_0)
	=\mathrm{Ann}_{R_\X}(dy_0)	=\langle 0\rangle$,
	since $x_0,y_0$ are non-zerodivisors of~$R_\X$.
\end{remark}

There are relations between
$\Omega^1_{R_\X/K}$ and $\Omega^1_{R_\X/R_o}$ as follows.

\begin{proposition}\label{Prop-S2-11}
Let $\X$ be an ACM set of $s$ distinct points
in~$\pmpn$. There exists an exact sequence of bigraded $R_\X$-modules
$$
0 \rightarrow R_\X dx_0\oplus R_\X dy_0
\hookrightarrow
\Omega^1_{R_\X/K} \stackrel{\psi}{\longrightarrow}
\Omega^1_{R_\X/R_o} \rightarrow 0
$$
where $\psi(gdf)=gd_{R_\X/R_o}f$ for $f,g\in R_\X$.
In particular, we have
$$
\HF_{\Omega^1_{R_\X/R_o}}(i,j)
= m\HF_\X(i-1,j)+ n\HF_\X(i,j-1)
+\HF_{\X}(i,j)-\HF_{\Y}(i,j)
$$
for all $(i,j)\in \N^2$, where $\Y$ is the subscheme
of $\pmpn$ defined by $I_\Y = I_{p_1}^2\cap\cdots\cap I_{p_s}^2$.
\end{proposition}
\begin{proof}
By \cite[Proposition 3.24]{Kun1986}, we have an
exact sequence of bigraded $R_\X$-modules
$$
R_\X \otimes_{R_o} \Omega^1_{R_o/K}
\stackrel{\varphi}{\longrightarrow}
\Omega^1_{R_\X/K} \stackrel{\psi}{\longrightarrow}
\Omega^1_{R_\X/R_o} \rightarrow 0
$$
where $\Omega^1_{R_o/K}\cong R_odx_0\oplus R_ody_0$
and $\varphi(f\otimes (f_1dx_0+f_2dy_0)) =  ff_1dx_0+ff_2dy_0$.
Hence the claimed exact sequence follows from
$\im(\varphi) = R_\X dx_0\oplus R_\X dy_0$.
Furthermore, the Hilbert function of~$\Omega^1_{R_\X/R_o}$ satisfies
$$
\HF_{\Omega^1_{R_\X/R_o}}(i,j)=
\HF_{\Omega^1_{R_\X/K}}(i,j)
- \HF_\X(i-1,j) - \HF_\X(i,j-1).
$$
An application of Theorem~\ref{Thm-S2-08} gives
the desired formula for $\HF_{\Omega^1_{R_\X/R_o}}$.
\end{proof}

\bigbreak
\section{The K\"{a}hler different}\label{Sec3}

Let $\X=\{p_1,...,p_s\} \subseteq \pmpn$ be a ACM set of points,
suppose that $\{x_0,y_0\}$ is a regular sequence in $R_\X$,
and let $R_o=K[x_0,y_0]$. Further, let $\{F_1,...,F_r\},$
$r\geq n+m$, be a bihomogeneous system of generators of $I_\X$.
By \cite[Corollary 2.14]{Kun1986}, $\Omega^1_{R_\X/R_o}$ has
the following presentation
$$
0\rightarrow \mathcal{K} \rightarrow
\textstyle{\bigoplus\limits_{i=1}^m} R_\X dX_i
\oplus \textstyle{\bigoplus\limits_{j=1}^n} R_\X dY_j
\rightarrow \Omega^1_{R_\X/R_o}\rightarrow 0
$$
where the bigraded $R_\X$-module $\mathcal{K}$ is generated
by the elements
$\sum_{i=1}^m \frac{\partial F_k}{\partial x_i}dX_i+\sum_{j=1}^n \frac{\partial F_k}{\partial y_j}dY_j$ for $j=1,...,r.$
The Jacobian matrix
$$
\mathcal{J} :=
\begin{pmatrix} \frac{\partial F_{1}}{\partial x_1}&\cdots&
	\frac{\partial F_{1}}{\partial x_m}&
	\frac{\partial F_{1}}{\partial y_1}&\cdots & \frac{\partial F_{1}}{\partial y_n}\\
	\vdots&\ddots& \vdots& \vdots&\ddots & \vdots\\
	\frac{\partial F_{r}}{\partial x_1}&\cdots& \frac{\partial F_{r}}{\partial x_m}&
	\frac{\partial F_{r}}{\partial y_1}&\cdots & \frac{\partial F_{r}}{\partial y_n}
\end{pmatrix}
$$
is a relation matrix of $\Omega^1_{R_\X/R_o}$ with respect to $\{dx_1,...,dx_m,dy_1,...,dy_n\}.$
It is easy to see that every $m+n$-minors of $\mathcal{J}$
is a bihomogeneous element of $R_\X$.

\begin{definition}
The bihomogeneous ideal of~$R_\X$ generated by all $m+n$-minors of
the Jacobian matrix $\mathcal{J}$ is called
the \textit{K\"{a}hler different} of $\X$ and is denoted
by $\vartheta_\X$.
\end{definition}

In the same way as above, we can define the K\"{a}hler different
$\vartheta_{\X_1}$ of $\X_1=\pi_1(\X)$ (or of the graded algebra
$R_{\X_1}/K[x_0]$). Similarly, we get the K\"{a}hler different
$\vartheta_{\X_2}$ of $\X_2 =\pi_2(\X)$ (or of the graded algebra
$R_{\X_2}/K[y_0]$).
When $|\X| = 1$, we see that $\vartheta_\X=\langle 1\rangle
=\vartheta_{\X_1}R_\X\cdot\vartheta_{\X_2}R_\X$.
In general, we have the following relation.

\begin{lemma} \label{Lem-S3-02}
\begin{enumerate}
	\item[(a)] We have
	$\vartheta_{\X_1}R_\X \cdot\vartheta_{\X_2}R_\X
	\subseteq \vartheta_\X$.
	\item[(b)] $\vartheta_\X$ contains a bihomogeneous non-zerodivisor.
\end{enumerate}
\end{lemma}
\begin{proof}
Obviously, we have $I_{\X_1}S \subseteq I_\X$ and
$I_{\X_2}S \subseteq I_\X$.
For any $G_{11},...,G_{1m}\in I_{\X_1}$ and
$G_{21},...,G_{2n}\in I_{\X_2}$, we have
$\{G_{11},...,G_{1m},G_{21},...,G_{2n}\}\subseteq I_\X$, and so
$$
\begin{aligned}
	\det&\begin{pmatrix} \frac{\partial G_{11}}{\partial x_1}&\cdots& \frac{\partial G_{11}}{\partial x_m}&
		\frac{\partial G_{11}}{\partial y_1}&\cdots & \frac{\partial G_{11}}{\partial y_n}\\
		\cdots&\cdots& \cdots& \cdots&\cdots & \cdots\\
		\frac{\partial G_{2n}}{\partial x_1}&\cdots&
		\frac{\partial G_{2n}}{\partial x_m}& \frac{\partial G_{2n}}{\partial y_1}&
		\cdots & \frac{\partial G_{2n}}{\partial y_n}
	\end{pmatrix}\\
	&= \frac{\partial (G_{11},...,G_{1m})}{\partial (x_1,...,x_m)} \cdot
	\frac{\partial (G_{21},...,G_{2n})}{\partial (y_1,...,y_n)}
	\in \vartheta_\X,
\end{aligned}
$$
where $\frac{\partial (G_{11},...,G_{1m})}{\partial (x_1,...,x_m)}$
denotes the image of the Jacobian determinant
$\frac{\partial (G_{11},...,G_{1m})}{\partial (X_1,...,X_m)}$
in~$R_\X$ (similarly for $\frac{\partial (G_{21},...,G_{2n})}{\partial (y_1,...,y_n)}$).
Moreover, $\vartheta_{\X_1}R_\X$ is generated by elements of the form
$\frac{\partial (G_{11},...,G_{1m})}{\partial (x_1,...,x_m)}$, and
$\vartheta_{\X_2}R_\X$ is generated by elements of the form
$\frac{\partial (G_{21},...,G_{2n})}{\partial (y_1,...,y_n)}$,
and therefore $\vartheta_{\X_1}R_\X \cdot\vartheta_{\X_2}R_\X
\subseteq \vartheta_\X$ and (a) follows.

To prove (b), observe that $x_0^iy_0^j \in R_\X$ is a bihomogeneous
non-zerodivisor for every $i,j\ge 0$.
By \cite[Proposition~3.5]{KLL2015}, there are $k,l\in \N$ such that $x_0^{k}\in \vartheta_{\X_1}$ and $y_0^l\in \vartheta_{\X_2}$.
Hence the non-zerodivisor $x_0^ky_0^l$ belongs to $\vartheta_\X$
by~(a).
\end{proof}

Some fundamental properties of the Hilbert function of $\vartheta_\X$
are given in the following proposition.

\begin{proposition} \label{Prop-S3-03}
Let $s_1=|\X_1|$ and $s_2=|\X_2|$.
\begin{enumerate}
	\item[(a)] For all $(i,j)\in \N^2$, we have
	$\HF_{\vartheta_\X}(i,j)\le \min\{\HF_{\vartheta_\X}(i+1,j),
	\HF_{\vartheta_\X}(i,j+1)\}$.

	\item[(b)] For all $i,j\in\N$, we have
	$\HF_{\vartheta_\X}(i,0)\le \HF_{\X_1}(i)$ and
	$\HF_{\vartheta_\X}(0,j)\le \HF_{\X_2}(j).$

	\item[(c)] If $s_1=1$ then $\HF_{\vartheta_\X}(i,j)=\HF_{\vartheta_{\X_2}}(j)$
	for all $(i,j)\in\N^2$; and if $s_2=1$ then $\HF_{\vartheta_\X}(i,j)=\HF_{\vartheta_{\X_1}}(i)$
	for all $(i,j)\in\N^2$.

	\item[(d)] For all $(i,j)\in\N^2$, we have
	$$
	\HF_{\vartheta_\X}(i,j)\le \HF_{\X}(i,j)\le
    \HF_{\vartheta_\X}(i+(m+1)(s_1-1),j+(n+1)(s_2-1)).
    $$
\end{enumerate}
\end{proposition}
\begin{proof}
Claim (a) follows by the fact that $x_0,y_0$ are non-zerodivisors
of $R_\X$ and $\vartheta_\X$ is a bihomogeneous ideal of $R_\X$.
Note that $\HF_{\vartheta_\X}(i,0)\le \HF_{\X}(i,0)$ and
$\HF_{\vartheta_\X}(0,j)\le \HF_{\X}(0,j)$ for all $i,j\in \N$.
So, claim (b) follows from \cite[Proposition~3.2]{Tuy2002}.

To prove (c), it suffices to consider the case $s_1=1$.
In this case we may assume $q_1=[1:0:...:0]\in\bbP^m$ and
$\X=\{q_1\times q'_1,....,q_1\times q'_s\}\subseteq \pmpn$.
We claim that
$I_\X =\langle X_1,...,X_m\rangle + I_{\X_2}S$.
Clearly, $\langle X_1,...,X_m \rangle + I_{\X_2}S\subseteq I_\X$.
Now let $F\in I_\X$ be homogeneous of degree $(i,j)$.
Using the Division Algorithm
(see e.g. \cite[Proposition~1.6.4]{KR2000}),
we may present $F=\sum_{k=1}^m H_kX_k+X_0^iG$ with
$H_k\in S_{i-1,j}$ and $G\in K[Y_0,...,Y_n]$ of degree $(0,j)$.
Then
$$
G(p_1\times q_l)= (X_0^iG)(p_1\times q_l)
=(F-\textstyle{\sum\limits_{k=1}^m} H_kX_k)(p_1\times q_l) =0
$$
for all $l=1,...,s$.
This implies $G\in I_{\X_2}$, and hence
$F\in \langle X_1,...,X_m \rangle + I_{\X_2}S$.

Consequently, the ideal $I_\X$ has a bihomogeneous system
of generators of the form $\{X_1,...,X_m, G_1,...,G_t\}$,
where $\{G_1,...,G_t\}$
is a homogeneous system of generators of
$I_{\X_2}\subseteq K[Y_0,...,Y_n]$.
Observe that $\vartheta_{\X_1}=\langle 1\rangle$
and $\vartheta_\X$ is generated by elements
$\frac{\partial(X_1,...,X_m, G_{k_1},...,G_{k_n})}{\partial (x_1,...,x_m,y_1,...,y_n)}
= \frac{\partial(G_{k_1},...,G_{k_n})}{\partial (y_1,...,y_n)}$
with $\{k_1,...,k_n\}\subseteq \{1,...,t\}$.
By Lemma~\ref{Lem-S3-02}(a), $\vartheta_\X = \vartheta_{\X_2}R_{\X}$.
Moreover, $R_\X \cong K[X_0, Y_0,...,Y_n]/I_{\X_2}
\cong R_{\X_2}[x_0]$.
Since $x_0$ is a non-zerodivisor of $R_\X$, we have
$$
\HF_{\vartheta_\X}(i,j) =
\HF_{\vartheta_{\X_2}R_\X}(i,j)
=\dim_K((\vartheta_{\X_2})_{j}x_0^i)
=\HF_{\vartheta_{\X_2}}(j)
$$
for all $(i,j)\in\N^2$.

For (d), it suffices to demonstrate the inequality
$$
\HF_{\X}(i,j)\leq\HF_{\vartheta_\X}(i+(m+1)(s_1-1),j+(n+1)(s_2-1)).
$$
In the proof of Lemma~\ref{Lem-S3-02}(b),
there exist $k,l\in \N$ such that
$h := x_0^ky_0^l\in \vartheta_{\X}$.
In particular, we may choose $k=(m+1)(s_1-1)$
and $l= (n+1)(s_2-1)$ by \cite[Proposition~3.5]{KLL2015}.
So, the multiplication map
$(R_\X)_{i,j}\stackrel{\times h}{\rightarrow}(\vartheta_\X)_{(i+k,j+l)}$
is injective as $K$-vector spaces.
This yields that $\HF_{\X}(i,j)\leq\HF_{\vartheta_\X}(i+k,j+l)$.
\end{proof}

The following corollary is a direct consequence of
Propositions~\ref{Prop-S2-04}(d) and~\ref{Prop-S3-03}(d).

\begin{corollary} \label{Cor-S3-04}
	In the setting of Proposition~\ref{Prop-S3-03},
	we have $\HF_{\vartheta_\X}(i,j)=s$ for all $(i,j)\succeq((s_1-1)(m+2),(s_2-1)(n+2)).$
\end{corollary}

\begin{lemma}\label{Lem-S3-05}
Let $\{h_1,...,h_t\}$ be a bihomogeneous minimal system of
generators of $\vartheta_\X$, write
$\deg(h_k)=(i_k,j_k)$ for $k=1,...,t$ and set
$$
i_{\max} := \max\{i_k \mid k=1,...,t\},\quad
j_{\max} := \max\{j_k \mid k=1,...,t\},
$$
and let $(i,j)\in\N^2$.
\begin{enumerate}
	\item[(a)]
	If $i\ge i_{\max}$ and
	$\HF_{\vartheta_\X}(i,j)=\HF_{\vartheta_\X}(i+1,j)$,
	then $\HF_{\vartheta_\X}(i,j)=\HF_{\vartheta_\X}(i+2,j)$.

	\item[(b)]
	If $j\ge j_{\max}$ and
	$\HF_{\vartheta_\X}(i,j)=\HF_{\vartheta_\X}(i,j+1)$,
	then $\HF_{\vartheta_\X}(i,j)=\HF_{\vartheta_\X}(i,j+2).$
\end{enumerate}
\end{lemma}
\begin{proof}
It suffices to prove (a), since (b) is analogous.
For $i\ge i_{\max}$, consider the multiplication map
$\mu_{x_0,i}:(\vartheta_\X)_{(i,j)}\rightarrow(\vartheta_\X)_{(i+1,j)},
h\mapsto x_0h$.
Since $\HF_{\vartheta_\X}(i,j)=\HF_{\vartheta_\X}(i+1,j)$,
$\mu_{x_0,i}$ is an isomorphism of $K$-vector spaces.
So, we have $(\vartheta_\X)_{(i+1,j)}=x_0\cdot(\vartheta_\X)_{(i,j)}.$
We need to show that $\mu_{x_0,i+1}:(\vartheta_\X)_{(i+1,j)}\rightarrow(\vartheta_\X)_{(i+2,j)}$
is also an isomorphism of $K$-vector spaces.
Clearly, $\mu_{x_0,i+1}$ is injective, as $x_0$ is a non-zerodivisor.
Now we check that $\mu_{x_0,i+1}$ is surjective.
Let $h\in (\vartheta_\X)_{(i+2,j)}\setminus\{0\}$.
Because $i\ge i_{\max}$, we may write
$h=\sum_{k=0}^m x_kg_k$ where $g_k \in (\vartheta_\X)_{i+1,j}.$
For each $k\in\{0,...,m\}$, we write $g_k=x_0g'_k$ for some
$g'_k\in (\vartheta_\X)_{(i,j)}$, and hence
$$
h = x_0g_0+\cdots+x_mg_m = x_0(x_0g'_0+\cdots+x_mg'_m)
\in x_0\cdot(\vartheta_\X)_{(i+1,j)}.
$$
Therefore $\mu_{x_0,i+1}$ is surjective, as wanted to show.
\end{proof}

From the lemma and the fact that $\HF_{\vartheta_\X}(i,j)\le s$
for all $(i,j)\in\N^2$, we have  $\HF_{\vartheta_\X}(i,j)=\HF_{\vartheta_\X}(i_{\max}+s,j)$
for all $i\ge i_{\max}+s$ and $j\in\N$
and $\HF_{\vartheta_\X}(i,j)=\HF_{\vartheta_\X}(i,s+j_{\max})$
for all $j\ge j_{\max}+s$ and $i\in\N$.

For $k,l \in \N$ set
$
\nu_k := \min\{ i\in\N \mid \HF_{\vartheta_\X}(i,k)= \HF_{\vartheta_\X}(i_{\max}+s,k)\}
$
and
$
\varrho_l := \min\{j\in\N \mid \HF_{\vartheta_\X}(l,j)= \HF_{\vartheta_\X}(l,j_{\max}+s)\}
$
and
$\nu_{\vartheta_\X}:=\sup\{\nu_k \mid k\in \N \}$ and
$\varrho_{\vartheta_\X} :=\sup\{\rho_l \mid l \in \N\}$.
Then $(\nu_{\vartheta_\X},\varrho_{\vartheta_\X}) \le (i_{\max}+s, j_{\max}+s)$
and if the values of $\HF_{\vartheta_\X}(i,j)$ for finite tuples
$(0,0)\preceq (i,j)\preceq
(\nu_{\vartheta_\X},\varrho_{\vartheta_\X})$ are computed,
then we know all values of $\HF_{\vartheta_\X}$.
This leads us to the following notion.

\begin{definition}
Let $(\nu,\varrho):=(\nu_{\vartheta_\X},\varrho_{\vartheta_\X})$.
The pair $B_{\vartheta_\X}=(B_{C,\vartheta_{\X}}, B_{R,\vartheta_{\X}})$, where
$$
B_{C,{\vartheta_\X}}=(\HF_{\vartheta_\X}(\nu,0), \HF_{\vartheta_\X}(\nu,1),\dots,
\HF_{\vartheta_\X}(\nu, \varrho))
$$
and
$$
B_{R,\vartheta_\X}=(\HF_{\vartheta_\X}(0,\varrho),\HF_{\vartheta_\X}(1,\varrho),\dots, \HF_{\vartheta_\X}(\nu,\varrho)),
$$
is called the \textit{border of the Hilbert function} of~$\vartheta_\X$.
\end{definition}

\begin{example}
Consider the set of nine points $\X \subseteq \bbP^2\times\bbP^2$
given in Example~\ref{Exam-S2-07}. We know that
$s_1=3$, $s_2=4$, $r_{\X_1}=1$, and $r_{\X_2}=2$.
Also, the set  $\X$ is ACM.
Then a bihomogeneous minimal system of generators
of $\vartheta_{\X}$ consists of 8 elements
with degrees in $\{(1,3),(2,2),(3,1),(0,5),(3,2)\}$.
This implies $i_{\max} = 3$ and $j_{\max}=5$.
The Hilbert function of $\vartheta_{\X}$ is computed by
$$
\HF_{\vartheta_{\X}}
= \begin{bmatrix}
0 & 0 & 0 & 0 & 0 & 1 & 1 & \cdots \\
0 & 0 & 0 & 1 & 2 & 2 & 2 & \cdots \\
0 & 0 & 3 & 8 & 9 & 9 & 9 & \cdots \\
0 & 1 & 6 & 8 & 9 & 9 & 9 & \cdots \\
0 & 1 & 6 & 8 & 9 & 9 & 9 & \cdots \\
\vdots&\vdots&\vdots&\vdots&\vdots&\vdots&\vdots&\ddots
\end{bmatrix}.
$$
It follows that $\nu_{\vartheta_\X}= i_{\max}=3$
and $\varrho_{\vartheta_\X} = j_{\max}=5$ and
the border of~$\HF_{\vartheta_{\X}}$
is $B_{\vartheta_\X} = ((0,1,6,8,9,9),(1,2,9,9))$.
\end{example}

If a bihomogeneous minimal system of $\vartheta_\X$
is given, we can compute the tuple $(\nu_{\vartheta_\X},\varrho_{\vartheta_\X})$ using
the following lemma.

\begin{lemma}
Let $\{h_1,...,h_t\}$ be a bihomogeneous minimal system of
generators of $\vartheta_\X$ with $\deg(h_k)=(i_k,j_k)$ for $k=1,...,t$. Put
$$
i_{\min} := \min\{i_k \mid k=1,...,t\},\quad
j_{\min} := \min\{j_k \mid k=1,...,t\}.
$$
Then $\nu_{\vartheta_\X} = \max\{\nu_{j_{\min}},\dots, \nu_{j_{\max}}\}$
and $\varrho_{\vartheta_\X} = \max\{  \varrho_{i_{\min}},\dots, \varrho_{i_{\max}}\}$.
\end{lemma}
\begin{proof}
For $(i,j)\in\N^2$ with $i< i_{\min}$ or $j< j_{\min}$,
it is clearly true that $\HF_{\vartheta_\X}(i,j) =0$.
By the definition of $\nu_j$ and $\nu_{\vartheta_\X}$,
we have $\nu_j=0$ if $j< j_{\min}$ and
$\nu_{\vartheta_\X} \ge \nu_{k}$ for $k\ge 0$.
It suffices to show that $\nu_{j_{\max}} \ge \nu_{k}$
for all $k\ge j_{\max}$.

When $k=j_{\max}$ and $i\ge \nu_{j_{\max}}$, we have $\HF_{\vartheta_{\X}}(i,k)
=\HF_{\vartheta_{\X}}(i+1,k)$.
So, $x_0(\vartheta_{\X})_{i,k}
=(\vartheta_{\X})_{i+1,k}$, since $x_0$ is a non-zerodivisor of~$R_\X$.
Also, for any $l\ge 0$, $(\vartheta_{\X})_{l,k+1}$ contains no minimal generators,
and hence $(\vartheta_{\X})_{l,k+1}=(\vartheta_{\X})_{l,k}\cdot (R_\X)_{0,1}$.
This implies $(\vartheta_{\X})_{i+1,k+1}
= (\vartheta_{\X})_{i+1,k}\cdot (R_\X)_{0,1}
= x_0(\vartheta_{\X})_{i,k}\cdot (R_\X)_{0,1}
= x_0 (\vartheta_{\X})_{i,k+1}$.
Thus $\HF_{\vartheta_{\X}}(i,k+1)=\HF_{\vartheta_{\X}}(i+1,k+1)$
for any $i\ge \nu_{j_{\max}}$, and so $\nu_{k} \ge \nu_{k+1}$.
By induction on $k$, we get $\nu_{j_{\max}} \ge \nu_{k}$
for all $k\ge j_{\max}$, and this completes the proof of the equality
for $\nu_{\vartheta_\X}$.
The equality for $\varrho_{\vartheta_\X}$ can be achieved similarly
using the non-zerodivisor $y_0\in (R_\X)_{0,1}$.
\end{proof}

As a consequence of the lemma, when $\vartheta_{\X}$ is a principal ideal
then $\nu_{\vartheta_\X} = \nu_{j_{\min}} = \nu_{j_{\max}}$
and $\varrho_{\vartheta_\X} = \varrho_{i_{\min}} = \varrho_{i_{\max}}$.

\bigbreak
\section{Special ACM Sets}\label{Sec4}

In this section we look at finite sets of points in~$\pmpn$
having the complete intersection or Cayley-Bacharach properties.
As before, we let $\X=\{p_1,...,p_s\}$ be a set of $s$ distinct
points in $\pmpn$.

\begin{definition}
\begin{enumerate}
\item[(a)]
$\X$ is called a \textit{complete intersection}
if its bihomogeneous ideal $I_\X$ is generated by a bihomogeneous
regular sequence.
\item[(b)]
If $I_\X$ is generated by $\{F_1,...,F_m, G_1,...,G_n\}$
which forms a bihomogeneous regular sequence with $F_i\in S_{d_i,0}$
and $G_j\in S_{0,d'_j}$ for $1\le i\le m$ and $1\le j\le n$,
we say that $\X$ is a \textit{complete intersection of type
$(d_1,...,d_m,d'_1,...,d'_n)$} and write $CI(d_1,...,d_m,d'_1,...,d'_n)$.
\end{enumerate}
\end{definition}

It is worth noticing that every complete intersection $\X$ is ACM.
When $\X=\X_1\times\X_2$, where $\X_k=\pi_k(\X)$
for $k=1,2$ (see Convention~\ref{Conv-S2-03}),
we also have the following property.

\begin{lemma}\label{Lem-S4-02}
Let $I_{\X_1}$, $I_{\X_2}$ be the homogeneous vanishing
ideals of $\X_1$ and $\X_2$, respectively.
If $\X = \X_1\times \X_2$,  then $\X$ is ACM with
$I_\X = I_{\X_1}S + I_{\X_2}S$ and
$$
\HF_\X(i,j)=\HF_{\X_1}(i)\cdot\HF_{\X_2}(j)
$$
for all $(i,j)\in\Z^2$.
\end{lemma}
\begin{proof}
The ACM proprety of $\X$ and the equality
$I_{\X_1}S + I_{\X_2}S = I_\X$ follow from
\cite[Theorem 2.1]{BK2002} and \cite[Lemma~3.5]{GJ2019}.
Moreover, we have $R_\X \cong R_{\X_1}\otimes_K R_{\X_2}$
by \cite[G.2]{Kun2005}, where $R_{\X_1}=K[x_0,...,x_m]/I_{\X_1}$
is the homogeneous coordinate ring of $\X_1\subseteq \bbP^m$
and $R_{\X_2}=K[y_0,...,y_n]/I_{\X_2}$
is the homogeneous coordinate ring of $\X_2\subseteq \bbP^n$.
Therefore we get the equality
$\HF_\X(i,j)=\HF_{\X_1}(i)\cdot\HF_{\X_2}(j)$
for all $(i,j)\in \Z^2$.
\end{proof}

As a direct consequence of the lemma, we
get the following shape of the border of the Hilbert function
of $\X$ for this case.

\begin{corollary}\label{Cor-S4-03}
In the setting of Lemma~\ref{Lem-S4-02}, let $s_k=|\X_k|$
and let $r_{\X_k}$ be the regularity index of $\HF_{\X_k}$
for $k=1,2$. The border $B_\X=(B_C,B_R)$
of the Hilbert function of $\X$ is given by
$$
B_C = (s_1,s_1\HF_{\X_2}(1),\dots, s_1\HF_{\X_2}(r_{\X_2})=s_1s_2)
$$
and
$$
B_R = (s_2,s_2\HF_{\X_1}(1),\dots, s_2\HF_{\X_1}(r_{\X_1})=s_1s_2).
$$
\end{corollary}

Notice that if $\X=\X_1\times\X_2$, then it is also ACM by
Lemma~\ref{Lem-S4-02}, so that the K\"{a}hler different of $\X$ exists.

\begin{proposition}\label{Prop-S4-05}
If $\X = \X_1\times\X_2$, then
the K\"{a}hler different $\vartheta_\X$ satisfies
$$
\vartheta_\X=\vartheta_{\X_1}R_\X \cdot\vartheta_{\X_2}R_\X.
$$
In addition, if $\X = CI(d_1,...,d_m,d'_1,...,d'_n)$,
then $\vartheta_\X$ is a bihomogeneous principal ideal
and has Hilbert function
$$
\HF_{\vartheta_\X}(r_{\X_1}+i, r_{\X_2}+j)=\HF_{\X}(i,j)
$$
for all $(i,j)\in\N^2$,
where $r_{\X_1}=\sum_{k=1}^md_k - m$ and
$r_{\X_2}=\sum_{l=1}^nd'_l - n$.
\end{proposition}
\begin{proof}
Suppose that $\{ F_1,...,F_r\}$ is a homogeneous system of generators of
$I_{\X_1}$ and $\{ G_1,...,G_t\}$ is a homogeneous system of generators
of~$I_{\X_2}$.
Then Lemma~\ref{Lem-S4-02} yields that
the relation matrix of $\Omega^1_{R/R_o}$ with respect to
$\{dx_1,...,dx_m,dy_1,...,dy_n\}$ is
$$
\begin{pmatrix}
\frac{\partial F_{1}}{\partial x_1}&\cdots&
\frac{\partial F_{1}}{\partial x_m}& 0&\cdots & 0\\
\vdots&\ddots& \vdots& \vdots&\ddots & \vdots\\
\frac{\partial F_{r}}{\partial x_1}&\cdots&
\frac{\partial F_{r}}{\partial x_m}& 0&\cdots & 0\\
0&\cdots& 0& \frac{\partial G_1}{\partial y_1}&\cdots &
\frac{\partial G_1}{\partial y_n}\\
\vdots&\ddots& \vdots& \vdots&\ddots & \vdots\\
0&\cdots& 0& \frac{\partial G_{t}}{\partial y_1}&\cdots &
\frac{\partial G_{t}}{\partial y_n}
\end{pmatrix}.
$$
Because
$\frac{\partial (F_{i_1},...,F_{i_k},G_{i_{k+1}},...,G_{i_{n+m}})}
{\partial (x_1,...,x_m,y_1,...,y_n)} = 0$
if $k\ne m$, it follows that
$\vartheta_\X$ is generated by elements of the form
$\frac{\partial (F_{i_1},...,F_{i_m},G_{j_1},...,G_{j_n})}
{\partial (x_1,...,x_m,y_1,...,y_n)}$
where $\{i_1,...,i_m\}\subseteq\{1,...,r\}$ and $\{j_1,...,j_n\}\subseteq\{1,...,t\}.$
But this element can be written as
$$
\frac{\partial (F_{i_1},...,F_{i_m},G_{j_1},...,G_{j_n})}
{\partial (x_1,...,x_m,y_1,...,y_n)}
=\frac{\partial (F_{i_1},...,F_{i_m})}{\partial (x_1,...,x_m)}\cdot
\frac{\partial (G_{j_1},...,G_{j_n})}{\partial (y_1,...,y_n)}.
$$
Hence we get $\vartheta_\X=\vartheta_{\X_1}R_\X\cdot\vartheta_{\X_2}R_\X$.
If $\X= CI(d_1,...,d_m,d'_1,...,d'_n) =\X_1\times\X_2$,
then $\X_1$ and $\X_2$ are complete intersections.
By \cite[Corollary 2.6]{KLL2015}, $\vartheta_{\X_1}$
is a principal ideal generated by a homogeneous
non-zerodivisor of degree $r_{\X_1}$ and
$\vartheta_{\X_2}$ is a principal ideal generated by a homogeneous
non-zerodivisor of degree $r_{\X_2}$, and hence
$\vartheta_{\X}$ is a principal ideal generated by a homogeneous
non-zerodivisor of degree $(r_{\X_1},r_{\X_2})$.
This also implies the claimed formula for $\HF_{\vartheta_\X}$.
\end{proof}

\begin{corollary}\label{Cor-S4-05}
If $\X = CI(d_1,...,d_m,d'_1,...,d'_n)$ and $B_\X=(B_C,B_R)$,
then we have $(\nu_{\vartheta_\X},\varrho_{\vartheta_\X})=(2r_{\X_1},2r_{\X_2})$
and the border of the Hilbert function $\vartheta_\X$ is given by
$$
B_{\vartheta_{\X}} = ((\underbrace{0,...,0}_{r_{\X_2}},
B_C),(\underbrace{0,...,0}_{r_{\X_1}},B_R)).
$$
\end{corollary}

Recall that for a finite set $\X$ of points in $\bbP^m$
and $p\in\X$, a \textit{minimal separator} of $p$ is a homogeneous
element $F\in K[X_0,...,X_m]$ of minimal degree such that
$F(p)\ne 0$ and $F(p') = 0$ for all $p'\in \X\setminus\{p\}$.
The \textit{degree} $\deg_\X(p)$ of $p$ in $\X$ is
the degree of a minimal separator of $p$.
We have $\deg_\X(p)\le r_\X$ for every point $p\in\X$, where
$r_\X$ is the regularity index of $\HF_\X$
(see \cite[Lemma~2.4]{GKR1993}).
We say that $\X$ is a \textit{Cayley-Bacharach scheme} if
all points of $\X$ have the same degree $r_\X$.
For many interesting results and more information
about these notions in the standard case,
see \cite{GKR1993, KLR2019}.

Now we look at the generalization of these notions for
a (not necessary ACM) set $\X$ of $s$ distinct points
in~$\pmpn$. In the same manner as above,
for each $p\in \X$, a bihomogeneous form
$F\in S$ is a \textit{separator} of $p$ in $\X$ if $F(p)\ne 0$
and $F(p')=0$ for all $p'\in\X \setminus\{p\}$, and a separator
$F\in S$ of $p$ in $\X$ is \textit{minimal} if there does not exist
a separator $G$ of $p$ with $\deg(G) \prec \deg(F)$.
For the existence of a finite set of minimal separators
of any point in $\X$ and their properties,
see e.g. \cite{GuVT2008, GuVT2008b, Mar2009}.

\begin{definition}
The \textit{degree} of a point $p\in \X$ is the set
$$
\deg_{\X}(p) = \{\deg(F )\mid F \
\textrm{is a minimal separator of}\ p\}.
$$
\end{definition}

For any $(i,j)\in\N^2$, we define $D_{(i,j)}:=\{(k,l)\in\N^2\mid (i,j)\preceq (k,l)\}$
and for a finite set $\Sigma=\{(i_1,j_1),...,(i_t,j_t)\}\subseteq\N^2$ we put
$D_{\Sigma}:=\bigcup_{k=1}^t D_{(i_k,j_k)}.$
Clearly, for every $(i,j)\in D_{\deg_\X(p)}$,
there exists a separator $F$ of $p$ with $\deg(F)=(i,j)$.
In the following we collect several useful properties
of degrees of points in $\X$ (see \cite[Theorem 5.7]{GuVT2008}
and \cite[Theorem 2.2]{GuVT2008b}).

\begin{theorem}\label{Thm-S4-07}
Let $p \in\X$ and $\Y=\X\setminus\{p\}$.
\begin{enumerate}
\item[(a)]
If $\{F_1,...,F_t\}$ is a set of minimal separators of $p$,
then $I_\Y = I_\X + \langle F_1,...,F_t\rangle.$
\item[(b)]
We have
\[
\HF_{\Y}(i,j)=\begin{cases}
	\HF_{\X}(i,j)&\ \textrm{if}\ (i,j)\notin D_{\deg_\X(p)}\\
	\HF_{\X}(i,j)-1&\ \textrm{if}\ (i,j)\in D_{\deg_\X(p)}.
\end{cases}
\]
\item[(c)] If $\X$ is ACM, then $|\deg_{\X}(p)|=1$
for every $p\in\X$.
\end{enumerate}
\end{theorem}

The converse of Theorem~\ref{Thm-S4-07}(c) holds true for $n=m=1$
by \cite[Theorem~6.7]{Mar2009}, but it fails to hold in general
(see \cite[Example 5.10]{GuVT2008} for an example
in $\bbP^2\times\bbP^2$). When $\X$ is ACM, we write
$\deg_\X(p)=(i,j)$ instead of $\deg_\X(p)=\{(i,j)\}.$

\begin{definition}
The set $\X$ is said to have the \textit{Cayley-Bacharach property}
if the Hilbert function of $\X\setminus \{p\}$ is independent
of the choice of $p\in \X$, or equivalently, if all of its
points have the same degree.
\end{definition}

In \cite[Proposition~7.3]{GKLL2017},
we know that $\X = CI(d_1,d'_1)$ if and only if
$\X$ has the Cayley-Bacharach property. However, it fails to hold in
general as the following example shows.

\begin{example}\label{Exam-S4-13}
	In $\bbP^1\times \bbP^2$, consider the set $\X=\X_1\times\X_2$
	of six points, where $\X_1 = \{q_1,q_2\} \subseteq \bbP^1$
	with $q_1=(1:0)$, $q_2=(1:1)$, and
	$\X_2=\{q'_1,q'_2,q'_3\} \subseteq \bbP^2$ with
	$q'_1=(1:0:0)$, $q'_2=(1:1:0)$, $q'_3=(1:1:1)$.
	Then $I_\X$ has a bihomogeneous minimal system of generators given by
	$$
	\{\, x_{0}x_{1} - x_{1}^{2}, y_{0}y_{1} - y_{1}^{2},
	y_{1}y_{2} - y_{2}^{2}, y_{0}y_{2} - y_{2}^{2} \,\},
	$$
	so $\X$ is not a complete intersection.
	On the other hand, $\X_1\subseteq \bbP^1$ is a complete intersection with $r_{\X_1}=1$,
	and hence $\X_1$ is a Cayley-Bacharach scheme,
	and $\X_2=\{q'_1,q'_2,q'_3\} \subseteq \bbP^2$ is also
	a Cayley-Bacharach scheme with $r_{\X_2}=1$.
	Using ApCoCoA we can check that $\deg(q_i\times q'_j)=(1,1)$
	for all $i=1,2$ and $j=1,2,3$.
	Thus $\X = \X_1\times\X_2$ has the Cayley-Bacharach property
	(this also follows by Proposition~\ref{Prop-S4-12}).
	In this case the K\"ahler different has its Hilbert function
	$$
	\HF_{\vartheta_\X} =
	\begin{bmatrix}
		0 & 0 & 0 & 0 & \cdots \\
		0 & 0 & 3 & 3 & \cdots \\
		0 & 0 & 6 & 6 & \cdots \\
		0 & 0 & 6 & 6 & \cdots \\
		\vdots & \vdots & \vdots & \vdots & \ddots
	\end{bmatrix}
	$$
	and $\HF_{\vartheta_\X}(r_{\X_1},r_{\X_2})
	=\HF_{\vartheta_\X}(1,1) =0$.
\end{example}

Using the K\"ahler different, we give a characterization of
complete intersections of type $(d_1,...,d_m,d'_1,...,d'_n)$
as follows.

\begin{theorem}\label{Thm-S4-10}
For a set $\X$ of $s$ distinct points in $\pmpn$,
the following statements are equivalent.
\begin{enumerate}
	\item[(a)] $\X = CI(d_1,...,d_m,d'_1,...,d'_n)$ for some
	positive integers $d_i,d'_j\ge 1$.
	\item[(b)] $\X = \X_1\times \X_2$ and $\X_1\subseteq \bbP^m$
	is a complete intersection of type $(d_1,...,d_m)$
	and $\X_2 \subseteq \bbP^n$ is a complete intersection of type $(d'_1,...,d'_n)$.
	\item[(c)] $\X = \X_1\times \X_2$ has the Cayley-Bacharach
	property and $\HF_{\vartheta_{\X}}(r_{\X_1},r_{\X_2})\ne 0$.
\end{enumerate}
\end{theorem}

In the proof of this theorem, we use the following properties.

\begin{lemma}\label{Lem-S4-11}
	For an ACM set of $s$ points $\X\subseteq \pmpn$,
	if $q\times q'\in\X$, then
	$$
	\deg_{\X}(q\times q')\preceq
	(\deg_{\X_1}(q),\deg_{\X_2}(q'))\preceq
	(r_{\X_1}, r_{\X_2}).
	$$
\end{lemma}
\begin{proof}
Since $\X$ is ACM, and so each point of $\X$ has exactly
one degree. The claim follows from the fact that
if $F_k$ is a separator of $q$ in~$\X_1$ and
$G_l$ is a separator of $q'$ in~$\X_2$,
then $F_kG_l$ is also a separator of $q\times q'$ in~$\X$.
\end{proof}

\begin{proposition}\label{Prop-S4-12}
Suppose $\X = \X_1\times \X_2 \subseteq \pmpn$.
Then $\X$ has the Cayley-Bacharach property
if and only if $\X_1$ and $\X_2$ are Cayley-Bacharach schemes.
\end{proposition}
\begin{proof}
Note that $\X$ is ACM.
Let us write $\X_1 =\{q_1,...,q_{s_1}\}\subseteq\bbP^m$
and $\X_2=\{q'_1,...,q'_{s_2}\}\subseteq\bbP^m.$
Firstly, we prove that
$$
\deg_{\X}(q_k\times q'_l)=(\deg_{\X_1}(q_k), \deg_{\X_2}(q'_l))
$$
for all $1\le k\le s_1, 1\le l\le s_2$.
According to Lemma~\ref{Lem-S4-11}, it suffices to show
that $\deg_{\X}(q_k\times q'_l) \succeq (\deg_{\X_1}(q_k), \deg_{\X_2}(q'_l))$.
Suppose $\deg_{\X}(q_k\times q'_l)=(i,j)$.
Let $F\in S_{i,j}$ be a minimal separator of the point
$q_k\times q'_l$. Then $F = \sum_u G_uH_u$ with $G_u\in S_{i,0}$
and $H_u\in S_{0,j}$. Let $T_1,...,T_{m_i}\in S_{i,0}$
(resp. $T'_1,...,T'_{n_j}\in S_{0,j}$)
be terms whose residue classes form a $K$-basis of
$S_{i,0}/(I_{\X_1}S)_{i,0}$ (resp. $S_{0,j}/(I_{\X_2}S)_{0,j}$).
This enables us to write
$G_u=a_{u1}T_1+\cdots+a_{um_i}T_{m_i}+ G'_u$
with $G'_u\in (I_{\X_1}S)_{i,0}$ and $a_{ur}\in K$,
$H_u=b_{u1}T'_1+\cdots+b_{un_j}T'_{n_j}+ H'_u$
with $H'_u\in (I_{\X_2}S)_{0,j}$ and $b_{ut}\in K.$
Since $I_\X = I_{\X_1}S+ I_{\X_2}S$, we have
$$
F = \sum_u G_uH_u =\!\!\!
\sum_{ 1\le r\le m_i, 1\le t\le n_j}c_{rt}T_{r}T'_{t}
\quad (\mathrm{mod}\ I_{\X}), \
\textrm{with} \ c_{rt}=\sum_u a_{ur}b_{ut}.
$$
Put $F_k := \sum_{rt} c_{rt}T'_t(q'_l)T_r \in S_{i,0}$.
Since $F(q_k\times q'_l)\ne 0$, we have $F_k(q_k)\ne 0$.
Moreover, $F_k(q_{k'})= F(q_{k'}\times q'_l)=0$ for $k'\ne k$.
So, $F_k$ is a separator of $q_k$ in~$\X_1$,
and this yields $i\ge \deg_{\X_1}(q_k)$.
Analogously, the element
$G_l := \sum_{rt} c_{rt}T_r(q_k)T'_t\in S_{0,j}$
is a separator of $q'_l$ in~$\X_2$, and hence
$j\ge  \deg_{\X_2}(q'_l).$
Thus, $(i,j) \succeq (\deg_{\X_1}(q_k),\deg_{\X_2}(q'_l))$,
and therefore we get
$\deg_{\X}(q_k\times q'_l)=(\deg_{\X_1}(q_k), \deg_{\X_2}(q'_l))$
for all $k,l$.

If $\X_1$ and $\X_2$ are Cayley-Bacharach schemes, then
$$
\deg_{\X}(q_k\times q'_l)=(\deg_{\X_1}(q_k), \deg_{\X_2}(q'_l))
= (r_{\X_1}, r_{\X_2})
$$
for all $1\le k\le s_1$ and $1\le l\le s_2$,
and hence $\X$ has the Cayley-Bacharach property.
Conversely, suppose that $\X$ has the Cayley-Bacharach property,
but $\X_1$ is not a Cayley-Bacharach-scheme.
Then there is a point $q_{k}\in\X_1$ such that $\deg_{\X_1}(q_{k})<r_{\X_1}$.
By \cite[Proposition 1.14]{GKR1993}, we find $q_{k'}\in\X_1$
such that $\deg_{\X_1}(q_{k'}) = r_{\X_1}$
and $q'_l\in\X_2$ such that $\deg_{\X_2}(q'_l)=r_{\X_2}$.
This implies
$$
\deg_{\X}(q_{k}\times q'_l)\preceq (r_{\X_1}-1, r_{\X_2})\prec
(r_{\X_1}, r_{\X_2})=\deg_{\X}(q_{k'}\times q_l),
$$
and thus $\X$ does not have the Cayley-Bacharach property,
a contradiction.
Therefore, $\X_1$ is a CB-scheme, so is $\X_2.$
\end{proof}

\begin{proof}[Proof of Theorem~\ref{Thm-S4-10}]
	The implication ``(b)$\Rightarrow$(a)'' follows from
	Lemma~\ref{Lem-S4-02}. To prove ``(a)$\Rightarrow$(b)'',
	suppose that $\X = CI(d_1,...,d_m,d'_1,...,d'_n)$ for some
	positive integers $d_i,d'_j\ge 1$.
	Then $I_\X = \langle F_1,...,F_m,G_1,...,G_n\rangle_S$
	with $\deg(F_i)=(d_i,0)$ and $\deg(G_j)=(0,d'_j)$,
	particularly, $I_{\X_1}=\langle F_1,...,F_m\rangle$
	is a saturated homogeneous ideal of~$K[X_0,...,X_m]$
	defining a complete intersection $\X_1 \subseteq \bbP^m$
	and $I_{\X_2} = \langle G_1,...,G_n\rangle$
	is a saturated homogeneous ideal of~$K[Y_0,...,Y_n]$ defining
	a complete intersection $\X_2 \subseteq \bbP^n$.
	Moreover, it is not hard to verify that $\X= \X_1\times \X_2$.

	The implication ``(b)$\Rightarrow$(c)'' holds true by
	Proposition~\ref{Prop-S4-05} and Proposition~\ref{Prop-S4-12}
	and the fact that a complete intersection set of $s$ points
	in $\bbP^m$ is always a Cayley-Bacharach scheme.

    Now we prove ``(c)$\Rightarrow$(b)''.
    It suffice to show that $\X_1$ is a complete intersection
in $\bbP^m$ (similarly for $\X_2\subseteq \bbP^n$).
By assumption, $\X$ has the Cayley-Bacharach property,
then $\X_1$ and $\X_2$ are Cayley-Bacharach schemes
by Proposition~\ref{Prop-S4-12}.
According to Proposition~\ref{Prop-S4-05}, we have
$\vartheta_\X=\vartheta_{\X_1}R_\X \cdot\vartheta_{\X_2}R_\X$,
and so  $\HF_{\vartheta_\X}(r_{\X_1},r_{\X_2})\ne 0$
implies $\HF_{\vartheta_{\X_1}}(r_{\X_1})\ne 0$.
By \cite[Theorem~5.6]{KL2017}, $\X_1$ is a complete intersection,
as desired.
\end{proof}

\begin{lemma}
If $\X=\X_1\times\X_2$ and for every point $p\in \X$
the K\"ahler different $\vartheta_{\X}$
contains no separator of $p$ of degree $\prec (mr_{\X_1}, nr_{\X_2})$,
then $\X$ has the Cayley-Bacharach property.
\end{lemma}
\begin{proof}
	Suppose that $\X$ does not have the Cayley-Bacharach property.
	By Proposition~\ref{Prop-S4-12}, $\X_1$ or $\X_2$ is not a Cayley-Bacharach
	scheme. Assume that $\X_1$ is not a Cayley-Bacharach scheme.
	There is $i\in\{1,...,s_1\}$ such that $\deg_{\X_1}(q_i) \le r_{\X_1}-1$.
	Let $F_i\in K[X_0,...,X_m]$ be a minimal separator of $q_i$ in $\X_1$
	and $F'_1\in K[Y_0,...,Y_n]$ be a minimal separator of $q'_1$ in $\X_2$.
	By \cite[Corollary~2.6]{KLL2015}, the image of $F_i^m$
	in $R_{\X_1}$ belongs to $\vartheta_{\X_1}$ and the image of
	${F'}_1^n$ in $R_{\X_2}$ belongs to $\vartheta_{\X_2}$.
	So, the image of $F_i^m{F'}_1^n$ in $R_\X$ is contained in~$\vartheta_{\X}$.
	Moreover, $F_i^m{F'}_1^n$ is a separator of $q_i\times q'_1$ in $\X$
	of degree $\preceq (m(r_{\X_1}-1), nr_{\X_2})$.
	This contradicts to the assumption.
\end{proof}

\bigbreak
\section{Finite Sets with the $(\star)$-Property}\label{Sec5}

Now we investigate the Cayley-Bacharach property for
a finite set $\X$ of points in $\pmpn$ which
satisfies the $(\star)$-property.
According to \cite[Definition~4.2]{GuVT2008},
the set $\X$ is said to have the \textit{$(\star)$-property}
if for any $q_1\times q'_1$, $q_2\times q'_2 \in \X$
then also either $q_1\times q'_2$ or $q_2\times q'_1\in \X$.
By \cite[Theorem~3.7]{GJ2019}, if $\X$ has the $(\star)$-property,
then $\X$ is ACM. Except for the case $m=n=1$, the converse of
this result does not hold true in general
(see \cite[Theorem 4.3 and Example 4.9]{GuVT2008}
and \cite[Example~4.2]{GJ2019}).
As before, for an ACM set $\X$
we always assume that $x_0,y_0$ form a regular sequence in~$R_\X$.

Write $\X_1 =\pi_1(\X) =\{q_1,...,q_{s_1}\}\subseteq\bbP^m$
and $\X_2=\pi_2(\X)=\{q'_1,...,q'_{s_2}\}\subseteq\bbP^m.$
For $i=1,...,s_1$ and $j=1,...,s_2$, put
$$
W_i := \pi_2(\pi_1^{-1}(q_i)\cap \X) \subseteq \X_2,
\quad
V_j := \pi_1(\pi_2^{-1}(q'_j)\cap \X) \subseteq \X_1.
$$
After renaming, we can always assume that
$|W_{s_1}|\le \cdots \le |W_1| \le s_2$ and
$|V_{s_2}|\le \cdots \le |V_1| \le s_1$.
When $\X$ has the $(\star)$-property,
we may assume $\X_1 = V_1 \supseteq\cdots\supseteq V_{s_2}$
and $\X_2 = W_1 \supseteq\cdots\supseteq W_{s_1}$
(see e.g. \cite[Lemma~3.4]{GJ2019}).

\begin{proposition}\label{Prop-S5-01}
If $\X$ has the $(\star)$-property, then
for $q_i\times q'_j \in \X$ we have
$$
\deg_\X(q_i\times q'_j) = (\deg_{V_j}(q_i), \deg_{W_i}(q'_j)).
$$
\end{proposition}
\begin{proof}
Since $\X$ is ACM, we have $\deg_\X(q_i\times q'_j)=(r,t)$
for some $(r,t)\in \N^2$. Clearly, $q_i\in V_j$ and
$q'_j \in W_i$. Let $G \in (K[X_0,...,X_m])_{\deg_{V_j}(q_i)}$
be a minimal separator of $q_i$ in $V_j$ and
$G'\in  (K[Y_0,...,Y_n])_{\deg_{W_i}(q'_j)}$
be a minimal separator of $q'_j$ in $W_i$.
Set $F := GG' \in S$.
Observe that $F(q_i\times q'_j) = G(q_i)G'(q'_j) \ne 0$.
Let $q\times q' \in \X \setminus\{q_i\times q'_j\}$.
If $q \in V_j\setminus\{q_i\}$ or $q'\in W_i\setminus\{q'_j\}$,
then $G(q)=0$ or $G'(q')=0$, and so $F(q\times q')=0$.
Now consider the case $q \notin V_j\setminus\{q_i\}$
and $q'\notin W_i\setminus\{q'_j\}$.
There are the following three cases:
\begin{itemize}
	\item If $q=q_i$ and $q'\ne q'_j$, then
	$q' \in W_i\setminus\{q'_j\}$, a contradiction.
	\item If $q'=q_j$ and $q\ne q_i$, then
	$q \in V_j\setminus\{q_i\}$, a contradiction.
	\item If $q\ne q_i$ and $q'\ne q'_j$,
	then the $(\star)$-property of $\X$ implies
	$q\times q'_j$ or $q_i\times q' \in \X$.
	It follows that $q \in V_j\setminus\{q_i\}$
	or $q' \in W_i\setminus\{q'_j\}$. This is again
	a contradiction.
\end{itemize}
Altogether, $F(q_i\times q'_j)\ne 0$ and $F(q\times q')=0$ for all
$q\times q' \in \X \setminus\{q_i\times q'_j\}$.
Hence $F$ is a separator of $q_i\times q'_j$
with $\deg(F) =  (\deg_{V_j}(q_i), \deg_{W_i}(q'_j))$,
and so $(r,t)\preceq (\deg_{V_j}(q_i), \deg_{W_i}(q'_j))$.

Furthermore, if $(r,t)\prec (\deg_{V_j}(q_i), \deg_{W_i}(q'_j))$,
then there is a minimal separator $\tilde{F}\ne 0$
of $q_i\times q'_j$ with $\deg(\tilde{F})=(r,t)$
and $r<\deg_{V_j}(q_i)$ or $t<\deg_{W_i}(q'_j)$.
Suppose that $r<\deg_{V_j}(q_i)$ (a similar argument for
the case $t<\deg_{W_i}(q'_j)$).
Set $\Y := V_j\times \{q'_j\} \subseteq \X$.
Then $\tilde{F}$ is also a separator of $q_i\times q'_j$ in~$\Y$.
As in the proof of Proposition~\ref{Prop-S4-12}, we have
$\deg_\Y(q_i\times q'_j) = (\deg_{V_j}(q_i),0)$.
This implies $(\deg_{V_j}(q_i),0) \preceq \deg(\tilde{F})=(r,t)$,
in particularly, we get $\deg_{V_j}(q_i) \le r < \deg_{V_j}(q_i)$,
a contradiction.
Therefore it must be $(r,t)= (\deg_{V_j}(q_i), \deg_{W_i}(q'_j))$.
\end{proof}

\begin{theorem}\label{Thm-S5-02}
Let $\X\subseteq\pmpn$ have the $(\star)$-property.
Then $\X$ has the Cayley-Bacharach property if and only if
the following conditions are satisfied:
\begin{enumerate}
	\item[(a)]  $V_1,...,V_{s_2}$ are Cayley-Bacharach schemes in $\bbP^m$
	and $r_{V_1}=\cdots = r_{V_{s_2}}$;
	\item[(b)] $W_1,...,W_{s_1}$ are Cayley-Bacharach schemes in $\bbP^n$
	and $r_{W_1}=\cdots = r_{W_{s_1}}$.
\end{enumerate}
\end{theorem}
\begin{proof}
If $\X$ satisfies the conditions (a) and (b), then
(a) implies $\deg_{V_j}(q)= r_{V_1}$ for all $q\in V_j$ and
for $j=1,...,s_2$, while (b) implies $\deg_{W_i}(q')= r_{W_1}$
for all $q'\in W_i$ and for $i=1,...,s_1$.
By Proposition~\ref{Prop-S5-01}, we obtain
$\deg(q\times q') = (r_{V_1}, r_{W_1})$ for all $q\times q'\in\X$.
Therefore $\X$ has the Cayley-Bacharach property.

Conversely, suppose that $\X$ has the Cayley-Bacharach property, i.e.,
there is $(r,t)\in \N^2$ such that
$\deg_\X(q\times q')=(r,t)$ for all $q\times q' \in \X$.
Note that we may here assume that
$\X_1 = V_1 \supseteq\cdots\supseteq V_{s_2}$
and $\X_2 = W_1 \supseteq\cdots\supseteq W_{s_1}$.
Especially, $\{q_1\}\times \X_2 \subseteq  \X$ and
and $\X_1 \times \{q'_1\} \subseteq \X$.
According to \cite[Proposition 1.14]{GKR1993},
$\X_1$ always contains a point $q_i$ of degree $r_{\X_1}$
and $\X_2$ always contains a point $q'_j$ of degree $r_{\X_2}$.
From $\deg(q_1\times q'_1) = \cdots = \deg(q_{s_1}\times q'_1) =(r,t)$,
Proposition~\ref{Prop-S5-01} yields
$$
r = \deg_{V_1}(q_1) = \cdots =\deg_{V_1}(q_{s_1}) =
\deg_{\X_1}(q_i) = r_{\X_1}.
$$
Similarly, it follows from
$\deg(q_1\times q'_1) = \cdots = \deg(q_1\times q'_{s_2}) =(r,t)$
and Proposition~\ref{Prop-S5-01} that
$$
t = \deg_{W_1}(q'_1) = \cdots =\deg_{W_1}(q'_{s_2}) =
\deg_{\X_2}(q'_j) = r_{\X_2}.
$$
In particular, $\X_1$ and $\X_2$ are Cayley-Bacharach schemes.
Moreover, we have $r_{V_{s_2}}\le\cdots \le r_{V_1}=r_{\X_1}$ and
$r_{W_{s_1}}\le\cdots \le r_{W_1}=r_{\X_2}$.
Thus $(r_{\X_1},r_{\X_2})=\deg_\X(q_i\times q'_j) = (\deg_{V_j}(q_i), \deg_{W_i}(q'_j)) \le (r_{V_j}, r_{W_i})$ for all $q_i\times q'_j \in\X$
implies $r_{V_{s_2}}=\cdots = r_{V_1}=r_{\X_1}$
and $r_{W_{s_1}}=\cdots = r_{W_1}=r_{\X_2}$ and
all $V_1,...,V_{s_2} \subseteq \bbP^m$
and  $W_1,...,W_{s_1} \subseteq \bbP^n$ are Cayley-Bacharach schemes.
\end{proof}

The next corollary is a direct consequence of Theorem~\ref{Thm-S5-02}.

\begin{corollary}
Let $\X\subseteq\pmpn$ have the $(\star)$-property.
If $\X$ has the Cayley-Bacharach property, then $\X_1\subseteq \bbP^m$
and $\X_2\subseteq \bbP^n$ are Cayley-Bacharach schemes.
\end{corollary}

\begin{example}
Let $K=\Q$ and $\X$ be the set of 24 points in
$\bbP^2\times \bbP^2$ given by
$\X = \X_1\times\X_2 \setminus \{q_5\times q_5\}$, where
$\X_1 = \X_2 = \{ q_1,...,q_5\}\subseteq \bbP^2$ with
$q_1=(1:0:0)$, $q_2=(1:1:0)$, $q_3=(1:0:1)$,
$q_4=(1:1:1)$ and $q_5=(1:1:2)$ (see the figure).
\begin{center}
	\includegraphics[scale=0.5]{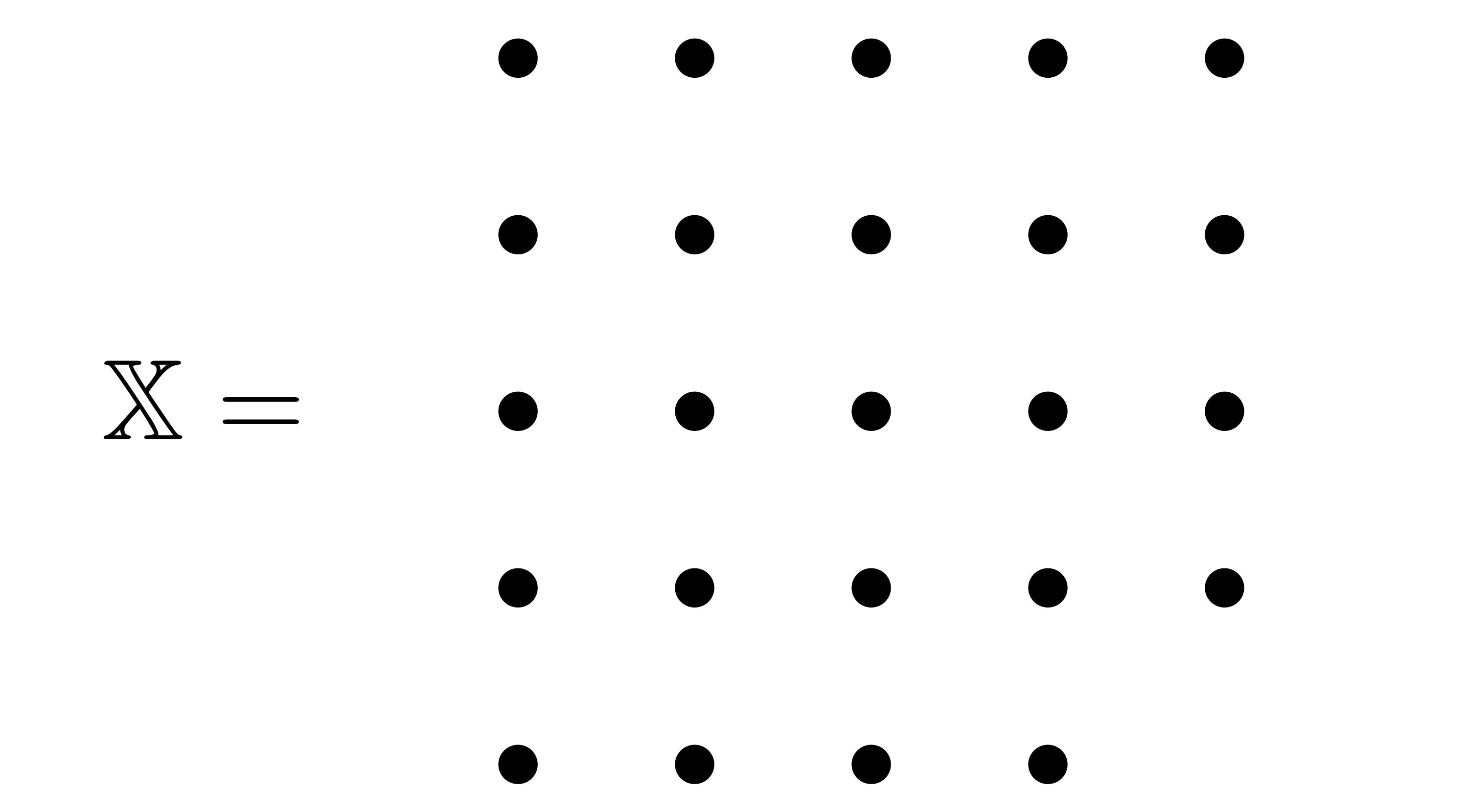}
\end{center}
Then we have $V_1 = V_2 = V_3 = V_4 = \X_1$, $V_5 = \X_1\setminus\{q_5\}$,
$W_1 = W_2=W_3=W_4=\X_2$ and $W_5 = \X_2\setminus\{q_5\}$.
Then $V_5$, $W_5$  are complete intersections in $\bbP^2$,
and so Cayley-Bacharach schemes. Also,
$\X_1$ is a Cayley-Bacharach scheme in $\bbP^2$ and
$r_{\X_1}= 2 = r_{V_5} = r_{W_5}$.
So, the conditions (a) and (b) in Theorem~\ref{Thm-S5-02}
are satisfied, and therefore $\X$ has the Cayley-Bacharach property.
\end{example}

\begin{proposition}\label{Prop-S5-05}
Let $\X\subseteq\bbP^1\times \bbP^n$ have the the $(\star)$-property.
Then $\X$ has the Cayley-Bacharach property
if and only if $\X = \X_1\times \X_2$ and
$\X_2\subseteq \bbP^n$ is a Cayley-Bacharach scheme.
\end{proposition}
\begin{proof}
Note that every finite set $V$ in $\bbP^1$ is a complete intersection
and $r_{V} = |V|-1$.
Suppose that $\X$ has the Cayley-Bacharach property.
Then Theorem~\ref{Thm-S5-02} yields $\X_1 = V_1 = \cdots = V_{s_2}$
and $\X_2 = W_1 \supseteq \cdots \supseteq W_{s_1}$ an descending chain
of Cayley-Bacharach schemes with $r_{\X_2}=r_{W_1}=\cdots=r_{W_{s_1}}$.
For $j=1,...,s_2$, we have
$\pi_1(\pi_2^{-1}(q'_j)\cap \X) = V_j =\{q_1,...,q_{s_1}\}$,
and so $\pi_2^{-1}(q'_j)\cap \X = \{q_1\times q'_j,...,q_{s_1}\times q'_j\}
\subseteq \X$. Hence $\X_1\times \X_2 \subseteq \X$, and therefore
$\X= \X_1\times \X_2$.
Conversely, assume that $\X= \X_1\times \X_2$ and $\X_2$ is a
Cayley-Bacharach scheme in $\bbP^n$.
Clearly, $\X_1\subseteq \bbP^1$ is a complete intersection,
and hence a Cayley-Bacharach scheme. By Proposition~\ref{Prop-S4-12},
$\X$ has the Cayley-Bacharach property.
\end{proof}

\begin{corollary}\label{Cor-S5-06}
Let $\X\subseteq\bbP^1\times \bbP^n$ have the $(\star)$-property.
Then the following statements are equivalent:
\begin{enumerate}
	\item[(a)] $\X = CI(d_1,d'_1,...,d'_n)$ for some positive integers
	$d_1,d'_1,...,d'_n\ge 1$.
	\item[(b)] $\X$ has the Cayley-Bacharach property
	and $\HF_{\vartheta_\X}(d_1-1,r_{\X_2})\ne 0$.
\end{enumerate}
\end{corollary}
\begin{proof}
This follows directly from Theorem~\ref{Thm-S4-10}
and Proposition~\ref{Prop-S5-05}.
\end{proof}

%
%

\medskip
\section*{Acknowledgements}
The authors thank Martin Kreuzer and Elena Guardo for
their encouragement to elaborate some results presented here.
The first two authors were partially supported by Hue University
under grant number DHH2021-03-159.
The last three authors were partially supported by
University of Education, Hue University.

\bigbreak
%
%


\begin{thebibliography}{99}

\bibitem{BK2002} S. Bouchiba and S. Kabbaj,
Tensor products of Cohen–Macaulay rings:
solution to a problem of Grothendieck,
J. Algebra \textbf{252} (2002), 65--73.

\bibitem{CN2020}
M. Chardin and N. Nemati, Multigraded regularity of complete intersections,
available at \texttt{arXiv:2012.14899v1} (2020).

\bibitem{DK1999} G. de Dominicis and M.~Kreuzer,
K\"ahler differentials for points in~$\mathbb{P}^n$,
J. Pure Appl. Alg.  \textbf{141}  (1999), 153--173.

\bibitem{GKR1993} A.V. Geramita, M. Kreuzer, and L. Robbiano,
Cayley-Bacharach schemes and their canonical modules,
Trans. Amer. Math. Soc. \textbf{339} (1993), 163--189.

\bibitem{GJ2019}
G. Favacchio and J. Migliore,
Multiprojective spaces and the arithmetically Cohen–Macaulay property,
Math. Proc. Camb. Phil. Soc. \textbf{166} (2019), 583--597.

\bibitem{GKLL2017} E. Guardo, M.~Kreuzer, T.~N.~K.~Linh,
and L.~N.~Long,  K\"ahler differentials for fat point schemes
in~$\mathbb{P}^1\times\mathbb{P}^1$, J. Commut. Algebra
(to appear 2021).

\bibitem{GLL2018} E. Guardo, T.~N.~K.~Linh,
and L.~N.~Long,
A presentation of the K\"ahler differential module for a fat point
scheme in~$\mathbb{P}^{n_1}\times\cdots\times\mathbb{P}^{n_k}$,
ITM Web of Conferences 20(4): 01007 (2018).

\bibitem{GuVT2004}  E.~Guardo and A. Van Tuyl,
Fat Points in $\mathbb{P}^1 \times \mathbb{P}^1$
and their Hilbert functions,
Canad. J. Math.  \textbf{56} (2004), no. 4, 716--741.

\bibitem{GuVT2008}  E.~Guardo and A. Van Tuyl,
ACM sets of points in multiprojective space,
Collect. Math. \textbf{59}(2) (2008), 191--213.

\bibitem{GuVT2008b}  E.~Guardo and A. Van Tuyl,
Separators of points in a multiprojective space,
Manuscripta Math. \textbf{126} (1) (2008), 99--113.

\bibitem{KLL2015}
M. Kreuzer, T.~N.~K.~Linh, and L.~N.~Long,
K\"{a}hler differentials and K\"{a}hler differents
for fat point schemes,
J. Pure Appl. Algebra \textbf{219} (2015), 4479--4509.

\bibitem{KL2017}
M. Kreuzer and Le N. Long,
Characterizations of zero-dimensional complete intersections,
Beitr. Algebra Geom. \textbf{58} (2017), 93–-129.

\bibitem{KLR2019}
M. Kreuzer, Le N. Long, and L. Robbiano,
On the Cayley-Bacharach property,
Communications in Algebra \textbf{47} (2019), 328--354.

\bibitem{KR2000}
M.~Kreuzer and L.~Robbiano,
Computational Commutative Algebra 1,
Springer Verlag, Heildelberg, 2000.

\bibitem{Kun1986}
E. Kunz, K\"{a}hler Differentials,
Adv. Lectures Math.,
Wieweg Verlag, Braunschweig, 1986.

\bibitem{Kun2005}
E. Kunz, Introduction to Plane Algebraic Curves,
Birkh\"auser, Boston, 2005.

\bibitem{Mar2009}
L. Marino, A characterization of ACM 0-dimensional subschemes of
$\bbP^1\times\bbP^1$, Le Matematische \textbf{LXIV} (2009), 41--56.

\bibitem{SS1975}
G. Scheja and U. Storch, \"Uber Spurfunktionen bei
vollst\"andigen Durchschnitten, J. Reine Angew. Math.
\textbf{278(279)} (1975), 174--190.

\bibitem{SVT2006}
J. Sidman and A. Van Tuyl, Multigraded regularity:
Syzygies and fat points,
Beitr\"age zur Algebra und Geometrie \textbf{47} (2006), 1--22.

\bibitem{Tuy2002}
A. Van Tuyl, The border of the Hilbert function of a set of points in
$\mathbb{P}^{n_1}\times\cdots\times\mathbb{P}^{n_k}$,
J. Pure Appl. Algebra \textbf{176} (2002), 223--247.

\bibitem{Tuy2003}
A. Van Tuyl, The Hilbert function of ACM set of points in
$\mathbb{P}^{n_1}\times\cdots\times\mathbb{P}^{n_k}$,
J. Algebra \textbf{264} (2003), 420-441.

\bibitem{ApC}
The ApCoCoA Team,
ApCoCoA: Applied Computations in Commutative Algebra,
available at \texttt{http://apcocoa.uni-passau.de}.

\end {thebibliography}
\end{document}